\documentclass[fontsize=11pt, headings=normal, headinclude=true, paper=a4, pagesize, cleardoublepage=current, BCOR=10mm, fleqn, numbers=noenddot]{scrartcl}
\pdfoutput=1
\usepackage{lmodern}

 \usepackage[english]{babel}  
  \usepackage[numbers,square]{natbib}
  \usepackage{mathrsfs}
  \usepackage{amsmath}
  \usepackage{amssymb}
  \usepackage{amsthm}
  \usepackage{enumitem}
  \usepackage[utf8x]{inputenc}    
  \usepackage[english=british]{csquotes} 

\usepackage[pdftex,plainpages=false,pdfpagelabels]{hyperref}                                
 	\hypersetup{
	pdftitle={Convergence of nodal sets in the adiabatic limit},
	pdfauthor={Jonas lampart},
     	colorlinks,
   		linkcolor=black,
   		filecolor=black,
   		urlcolor=black,
   		citecolor=black,
   		plainpages=false,
}

\newtheorem{thm}{Theorem}[section]
\newtheorem{prop}[thm]{Proposition}
\newtheorem{lem}[thm]{Lemma}
\newtheorem{cor}[thm]{Corollary}

\newtheorem*{thm*}{Theorem}
\newtheorem*{Kor*}{Korollar}

\theoremstyle{definition}

\newtheorem{rem}[thm]{Remark}

\newtheorem*{cond*}{Condition}

\newcommand{\norm}[1]{\ensuremath{\lVert #1 \rVert}}
\newcommand{\abs}[1]{\ensuremath{\lvert #1 \rvert}}
\newcommand{\eps}{\varepsilon}

\newcommand{\ud}{\mathrm{d}}

\newcommand{\ui}{\mathrm{i}}
\newcommand{\R}{\mathbb{R}}

\newcommand{\N}{\mathbb{N}}
\newcommand{\eff}{\mathrm{eff}}

\DeclareMathOperator{\vol}{vol}

\DeclareMathOperator{\grad}{grad}
\DeclareMathOperator{\divg}{div}
\DeclareMathOperator{\dist}{dist}

\DeclareMathOperator{\tr}{tr}
\DeclareMathOperator{\supp}{supp}

\newlist{teile}{enumerate}{2}
\setlist[teile,1]{label=\arabic*),fullwidth}
\setlist[teile,2]{label=(\alph*), fullwidth}

\begin{document}
\title{Convergence of nodal sets in the adiabatic limit}
\author{Jonas Lampart
\thanks{CEREMADE, Universit\'e Paris-Dauphine; \textsc{lampart@ceremade.dauphine.fr}}}
\maketitle
\begin{abstract}
 We study the nodal sets of non-degenerate eigenfunctions of the Laplacian on fibre bundles $\pi{:}\, M\to B$ in the adiabatic limit. This limit consists in considering a family $G_\varepsilon$ of Riemannian metrics, that are close to Riemannian submersions, for which the ratio of the diameter of the fibres to that of the base is given by $\varepsilon \ll 1$.
 
We assume $M$ to be compact and allow for fibres $F$ with boundary, under the condition that the ground state eigenvalue of the Dirichlet-Laplacian on $F_x$ is independent of the base point. We prove for $\mathrm{dim}(B) \leq 3$ that the nodal set of the Dirichlet-eigenfunction $\varphi$ converges to the pre-image under $\pi$ of the nodal set of a function $\psi$ on $B$ that is determined as the solution to an effective equation. In particular this implies that the nodal set meets the boundary for $\varepsilon$ small enough and shows that many known results on this question, obtained for some types of domains, also hold on a large class of manifolds with boundary.
For the special case of a closed manifold $M$ fibred over the circle $B=S^1$ we obtain finer estimates and prove that every connected component of the nodal set of $\varphi$ is smoothly isotopic to the typical fibre of $\pi{:}\, M\to S^1$.
\end{abstract}

\section{Introduction}
For an eigenfunction $\varphi$ of a differential operator on the manifold $M$ the nodal set is defined as $\mathcal{N}(\varphi):=\overline{\varphi^{-1}(0)\cap(M\setminus \partial M)}$. The complex patterns that this set forms were discovered by Chladni~\cite{Chlad}, who made them visible by letting sand settle at those points on a vibrating plate where it was at rest.
Subsequently many aspects of the nodal sets have been studied by mathematicians. One of the basic concepts related to the nodal set is that of a nodal domain, a connected component of $M\setminus \mathcal{N}(\varphi)$. The fundamental theorem on nodal domains is originally due to Courant and gives an upper bound on the number of nodal domains (see~\cite[chapter 6]{SY}).
\begin{thm*}[Courant's nodal domain theorem]
 Let $(M,G)$ be a compact, connected Riemannian manifold with boundary. Let $0\leq\lambda_0< \lambda_1\leq \dots$ be the eigenvalues  of $-\Delta_G$ with Dirichlet boundary conditions, repeated according to multiplicity. If $\varphi_k$ is an eigenfunction corresponding to $\lambda_k$, the number of nodal domains of $\varphi_k$ is at most $k+1$.
\end{thm*}
One is immediately drawn to ask when this bound is sharp. Certainly the ground state $\varphi_0$ has exactly one nodal domain, and since $\varphi_0$ may be chosen real and positive $\varphi_1$ must have exactly two. On the contrary Pleijel showed that for domains in $\R^2$ the bound can only be attained for a finite number of eigenvalues~\cite{Plj}.
A more specific question is the relation of $\mathcal{N}(\varphi)$ to the boundary. Payne conjectured~\cite{Pay}, again for domains in $\R^2$, that the nodal line of $\varphi_1$ joins two points on the boundary. This conjecture has been proven by Melas for convex domains~\cite{Mel} and other sufficient conditions were found by M. Hoffmann-Ostenhof, T. Hoffmann-Ostenhof and Nadirashvili~\cite{H2Oconj}.
For general domains however these authors found a counter example~\cite{H2Oconj}. Such counter examples were later given by Fournais~\cite{Four} for domains in $\R^d$, $d\geq 2$ and by Freitas~\cite{Fr} for the unit disc with a non-Euclidean metric.

More detailed and quantitative results can be obtained for special types of domains. For instance Jerison~\cite{JerDiam} and Grieser-Jerison~\cite{GrJe} were able to obtain estimates on the location of the nodal set of $\varphi_1$, implying in particular that it meets the boundary, for convex two-dimensional domains of large eccentricity. For such domains in higher dimensions Jerison~\cite{Jer} also proved that the nodal set of $\varphi_1$ touches the boundary.
Similar ideas were used by the same authors to estimate the location and size of the maximum of $\varphi_0$~\cite{GrJeSize}.
Freitas and Krej{\v{c}}i{\v{r}}{\'\i}k~\cite{FrKr} considered a different type of \enquote{thin} domains. These are given as embeddings of $[0,1]\times \Omega$ into $\R^k$, where $\Omega \subset \R^{k-1}$ is a compact domain, such that the image of $\Omega$ has diameter $\eps\ll 1$. They show that there is a number $N(\eps)$, increasing as $\eps \to 0$, of eigenfunctions that attain Courant's bound, and that the nodal sets of these eigenfunctions touch the boundary.
In a recent contribution~\cite{KTnodal}, which appeared after the preprint of the present article, Krej{\v{c}}i{\v{r}}{\'\i}k and Tu{\v{s}}ek proved the same result for thin tubular neighbourhoods of codimension one hypersufraces in $\R^d$.
The results~\cite{FrKr,GrJe,  GrJeSize, JerDiam, Jer, KTnodal} all rely on analysing the problem in an asymptotic situation, where in fact the behaviour of $\varphi$ can be determined using the solution to an effective equation, which is an equation on the unit interval (except in~\cite{KTnodal}).

In the present paper we will pursue similar ideas, in that we consider \enquote{thin} fibre bundles $\pi{:}\,M\to B$ and show how to determine the behaviour of $\mathcal{N}(\varphi)$ using an effective equation on the base $B$. This will allow us to obtain results in the spirit of~\cite{FrKr,GrJe, JerDiam, Jer} for a large class of compact manifolds, both closed and with boundary. In the latter case this answers the question, posed by Schoen and Yau~\cite[problem 45]{SY}, whether such ideas apply to manifolds with boundary. The case of closed manifolds allows for even more precise results, at least in the case $B=S^1$, where we are able to not only locate the nodal set but also determine it up to smooth isotopy.

\section{Nodal sets in the adiabatic limit}
Let $M$, $F$ be a compact, connected manifolds with boundary and $B$ a closed (i.e.~compact without boundary) and connected manifold. Let
\begin{equation*}
 \pi{:}\, M\to B
\end{equation*}
be a smooth fibre bundle with fibre $F$, that is for every $x\in B$ there exists a neighbourhood $U$ of $x$ such that $\pi^{-1}(U)$ is diffeomorphic to $U\times F$. 
Let $g$, $g_B$ be metrics of $M$ and $B$ respectively such that the differential $\pi_*$ induces an isometry $\pi_*{:}\, TM/\ker \pi_* \to TB$. Then we can write 
\begin{equation*}
g=g_F + \pi^*{g_B}\,,
\end{equation*}
where $g_F$ vanishes on horizontal vectors, that is on the $g$-orthogonal complement of $TF:=\ker \pi_*$.
On every fibre $F_x:=\pi^{-1}(x)$, $g_F$ is just given by the restriction $g\vert_{TF}$ to the tangent space of $F_x$, which we call the vertical subspace of $TM$. The adiabatic limit of $(M, g)$ is defined to be the family $(M, g_\eps)$, $0<\eps<1$, with
\begin{equation}\label{eq:geps}
 g_\eps:=g_F + \eps^{-2}\pi^*g_B\,.
\end{equation}
Clearly the diameter of $B$ grows like $\eps^{-1}$, while that of the fibres is fixed, and so the fibres become thin relative to the base.

In order to account also for metrics that arise from shrinking families of embeddings as discussed in~\cite{FrKr} we will consider slightly more general metrics $G_\eps$, which are perturbations of $g_\eps$ and whose exact form we discuss in section~\ref{sect:pert}. 
These metrics will lead us to an operator of the form
$-\Delta_{g_\eps} + \eps H_1$, which is unitarily equivalent to $-\Delta_{G_\eps}$. 
In order to understand this operator we first note that the dual metric to $g_\eps$ on $T^*M$ (which we denote by the same symbol) acts on $\pi^*\xi \in \pi^*T^*B$ as $g_\eps(\pi^*\xi,\pi^*\xi)=\eps^2g_B(\xi, \xi)$.
Hence we have
\begin{equation*}
 \Delta_{g_\eps}=\eps^2 \Delta_h + \Delta_F\,,
\end{equation*}
 where $\Delta_F$ is the Laplace-Beltrami operator of the metric restricted to the fibres and 
\begin{equation*}
 \Delta_h:=\tr_{\pi^*g_B} \nabla^2 - \eta
\end{equation*}
is a horizontal differential operator determined by the Levi-Cività connection $\nabla$ and the mean curvature vector $\eta$ of the fibres, both with respect to $g=g_{\eps=1}$. Possibly the most important feature of this decomposition is that the vertical part $\Delta_F$ is independent of $\eps$.

From now on we will assume:
\begin{cond*}
The ground state energy of the fibre Laplacian with Dirichlet conditions
\begin{equation*}
 \Lambda_0(x):=\min_{0\neq\phi\in W^1_0(F_x, g_{F_x})} \norm{\phi}_{L^2(F_x, g_{F_x})}^{-2} \int_{F} g_{F_x}(\ud\phi, \ud \phi)\,\mathrm{vol}_{g_{F_x}} 
\end{equation*}
is constant, i.e~does not depend on the base point $x\in B$.
\end{cond*}
This is of course always true if $M$ is closed, since then $\Lambda_0\equiv 0$. 
We define 
\begin{equation*}
H:=-\Delta_{g_\eps} + \eps H_1 - \Lambda_0\,,
\end{equation*}
with Dirichlet boundary conditions (i.e.~on the domain $D(H):= W^2(M)\cap W^1_0(M)$, where we denote by $W^k(M)$ the $L^2$-Sobolev space $W^{k,2}(M,g)$, which is independent of $g$ as a topological vector space so we will only make the dependence on the metric explicit if we use a specific norm). This operator and its eigenfunctions will be our object of study. 

In our recent work in collaboration with Teufel~\cite{LT} we developed an approximation scheme for operators of this form.
For the case at hand this gives us an effective operator on $L^2(B)$, which can be used to obtain approximations of eigenvalues and eigenfunctions of $H$ for small $\eps$ (see theorem~\ref{thm:adiab} for the precise statement).
 Just as in the earlier works~\cite{FrKr,GrJe,  GrJeSize, JerDiam, Jer, KTnodal} it will be these approximate eigenfunctions that allow us to locate the nodal sets of the true eigenfunctions. 
More precisely, we will find an $\eps$-independent operator $H_0$ on $B$ whose eigenvalues $\mu$ are in one-to-one correspondence with the eigenvalues of $H$.
In particular if $\mu$ is a simple eigenvalue of $H_0$, then there exists a unique eigenvalue $\lambda=\eps^2 \mu + \mathcal{O}(\eps^3)$ of $H$, and it is also simple. 
Our main result is:
\begin{thm*}
Assume $\dim B\leq 3$ and that $G_\eps$ satisfies the conditions specified in section~\ref{sect:pert}. Let $H_0$ be the self-adjoint operator with domain $W^2(B)\subset L^2(B)$ given by~\eqref{eq:Ha} and denote by $\phi_0$ the positive ground state of $-\Delta_F$. 

Then for every simple eigenvalue $\mu\in \sigma(H_0)$ with normalised eigenfunction $\psi\in \ker(H_0-\mu)$ and $\eps$ small enough, there exist $\lambda=\eps^2 \mu + \mathcal{O}(\eps^3)\in \sigma(H)$ and $\varphi \in \ker (H - \lambda)$ such that $\varphi\to\pi^*\psi\phi_0$ uniformly as $\eps\to 0$.
If zero is a regular value of $\psi$, then $\varphi$ has at least as many nodal domains as $\psi$ and $\mathcal{N}(\varphi)$ converges to $\pi^{-1}\mathcal{N}(\psi)$ in Hausdorff distance. If also $\partial M\neq \varnothing$, then $\mathcal{N}(\varphi)\cap \partial M \neq \varnothing$.

In the special case $B=S^1$ and $\partial M=\varnothing$ the nodal set of $\varphi$ consists of finitely many submanifolds $F\hookrightarrow M$ and is smoothly isotopic through embeddings to $\pi^{-1}\mathcal{N}(\psi)$.
\end{thm*}
The proof of this theorem is given via several, slightly more precise and quantitative, statements.
These are: theorem~\ref{thm:uniform} for the convergence of eigenfunctions, theorem~\ref{thm:node} for the estimate on the nodal count, proposition~\ref{prop:boundary} for the convergence of nodal sets and theorem~\ref{thm:iso} for the final statement.

The condition $\dim B\leq 3$ is due to our technique of proving uniform convergence of eigenfunctions. To do this we use the Sobolev embedding theorem $W^1(B)\hookrightarrow \mathscr{C}^0(B)$ in the case $\dim B=1$. In higher dimensions a similar technique allows us to trade the possible lack of regularity of $W^1(B)$ for a slower speed of convergence. This unfortunately gives no meaningful estimate if $\dim B>3$. In the recent article~\cite{KTnodal} the authors were able to obtain some results without restriction on the base dimension. Their technique involves proving convergence $\varphi\to \pi^*\psi \phi_0$ in Sobolev spaces of arbitrary order. This is facilitated by the fact that, in the setting they consider, the vertical and horizontal part of $\Delta_{g_\eps}$ commute, which is not the case in general.
The sharper estimates in the case $\dim B=1$ also allow for good control of the derivatives of the eigenfunctions, which makes it possible to prove that $\mathcal{N}(\varphi)$ is essentially a graph over, and hence isotopic to, $\pi^{-1}\mathcal{N}(\psi)$.

Our approach, in particular the results of~\cite{LT}, also apply to the case of a varying ground state energy.  
However, if for example $\Lambda_0$ has a non-degenerate minimum on $B$, the behaviour of the eigenvalues and eigenfunctions of $H$ is quite different. In this case the small eigenvalues of $H$ and the corresponding eigenfunctions behave like those of an harmonic oscillator. In particular the typical distance between the eigenvalues is of order $\eps$, as opposed to $\eps^2$ for constant $\Lambda_0$, and the eigenfunctions are exponentially localised in a neighbourhood of size $\sqrt \eps$ of the minimum of $\Lambda_0$. Hence in order to obtain non-trivial results one must blow up this neighbourhood and perform a  detailed analysis of the eigenfunctions there. Some results on this case are discussed in~\cite{Lam} and we will further analyse this in a future paper.

An interesting special case is given by thin tubes around embedded submanifolds $B\hookrightarrow \R^n$  (For embeddings into Riemannian manifolds our conditions on the induced metric $G_\eps$ will not be satisfied in general, see however~\cite[remark 3.6]{Lam}). These are often called quantum waveguides and admit many generalisations, including surfaces of such tubes, as discussed in~\cite{HLT}.
To illustrate this, let $T_\eps$ be the closed tubular neighbourhood of $\iota{:}\, B\hookrightarrow \R^n$ with radius $\eps$. Then $T_\eps$ is diffeomorphic to the fibre bundle $M:=\lbrace v\in NB: \abs{v}^2 \leq 1 \rbrace \stackrel{\pi}{\to} B$, via $v\mapsto \Phi_\eps(v):=\iota(\pi(v))+\eps v$, where $NB\subset T\R^n$ denotes the normal bundle and $\R^n$ carries the standard metric $\delta$. On $M$ we choose the metric
$G_\eps:=\eps^{-2} \Phi^*_\eps\delta$. This has the form  $G_\eps=g_F+\eps^{-2}(\iota^*\delta+\mathcal{O}(\eps))$, where $g_F$ is the standard metric on the unit ball (see e.g.\cite{HLT}, where explicit expressions for $G_\eps$ and the operators $H_0$ and $H_1$ are derived). Hence the fibres of $(M, G_\eps)$ are all isometric and thus $\Lambda_0$ is constant. We may thus apply our main theorem on $M$ and map the resulting statements to $T_\eps$ to obtain:
\begin{cor}
 Let $\iota{:}\, B\to \R^n$ with $\dim B\leq 3$ and $T_\eps$ be as above. Let $H_0$ and be as in the theorem and $\mu\in \sigma(H_0)$ a simple eigenvalue. 

Then for $\eps$ small enough there exists a unique eigenvalue $\lambda=\mu +\eps^{-2}\Lambda_0 +\mathcal{O}(\eps)$ of the Dirichlet Laplacian in $T_\eps$. If furthermore zero is a regular value of $\psi\in \ker(H_0-\mu)$ (hence in particular if $\dim B=1$) then any eigenfunction $\varphi$ corresponding to $\lambda$ has at least as many nodal domains as $\psi$, its nodal set intersects the boundary of $T_\eps$ and converges to $\mathcal{N}(\psi)$ in Hausdorff distance. 
\end{cor}

This corollary is parallel to the results of~\cite{FrKr} (where $B=[0,1]$, which is not allowed in our case) and~\cite{KTnodal} (where the dimension of $B$ is arbitrary but the codimension is one).
Hence in~\cite{FrKr}, and also in~\cite{GrJe, JerDiam, Jer}, the operator corresponding to $H_0$ is an operator on a closed interval $I\subset \R$ with Dirichlet conditions. Thus all its eigenvalues are simple and their eigenfunctions attain Courant's bound, which then forces the eigenfunctions of the operator on $M$ to attain the bound as well. For a general base manifold this is of course not always the case. However the argument still applies to the point that if $\psi$ attains Courant's bound, then so must $\varphi$. This can be used to construct many examples of eigenfunctions that do so.
\subsection{\textbf{The perturbed metric $G_\eps$}}\label{sect:pert}
In this section we specify the exact form of the perturbed metric $G_\eps$ that we consider, and translate this into conditions on $\eps H_1$. For $0<\eps< 1$ let $G_\eps$ be a family of Riemannian metrics on $M$ that satisfies
\begin{equation*}
\abs{\big(G_\eps - g_\eps\big)(v,v)}\leq \eps^{-1}\pi^*g_B(v,v)\,,
\end{equation*}
for all $v\in TM$. In particular this means that the difference of these metrics vanishes on $TF$ and for horizontal vectors it is bounded by $\eps g_\eps$. Note also that by polarisation we have for every $w\in TF$ and $v\in TM$
\begin{equation*}
 \abs{\big(G_\eps - g_\eps\big)(v,w)}
 \leq\tfrac12 \eps^{-1}g_B(\pi_*v, \pi_* v)\,.
\end{equation*}
This implies vanishing of the left hand side of the inequality by scaling $w$, so the space of horizontal vectors does not depend on $\eps$.
Additionally we assume that all covariant derivatives (with respect to $g_\eps$) of the difference are bounded by $\eps g_\eps$.
For a detailed discussion of how such metrics arise from embeddings we refer to~\cite{HLT}. 
In order to bring $-\Delta_{G_\eps}$ into the form $-\Delta_{g_\eps} + \eps H_1$ we define a local unitary
\begin{equation*}
  U_\rho{:}\, L^2(M, G_\eps)\to L^2(M, g_\eps)\,,\qquad 
f\mapsto \sqrt{\frac{\mathrm{vol}_{G_\eps}}{\mathrm{vol}_{g_\eps}}}f=:\sqrt{\rho}f\,.
\end{equation*}
Transformation by this unitary gives (see e.g.~\cite[lemma 1]{WaTeConst} for details)
\begin{equation*}
- U_\rho \Delta_{G_\eps} U_\rho^* = -\Delta_{g_\eps} + \eps H_1
\end{equation*}
with $\eps H_1= \eps^3 S_\eps + V_\rho$, where
 \begin{equation*}\label{eq:Seps}
 S_\eps f=\eps^{-3} \divg_{g_\eps}\Big(\big(g_\eps - G_\eps\big)(\ud f, \cdot)\Big)
\end{equation*}
and
\begin{equation*}\label{eq:Vrho}
 V_\rho=\tfrac12 \divg_{g_\eps}\grad_{G_\eps}(\log \rho) + \tfrac14 G_\eps(\ud \log \rho, \ud \log \rho)\,.
\end{equation*}
Note that $S_\eps{:}\, W^2(M, g)\to L^2(M,g)$ is of order one since for $\xi \in T^*B$
\begin{equation}\label{eq:g diff}
\abs{\big(g_\eps - G_\eps\big)(\pi^*\xi, \pi^*\xi)} \leq \eps g_\eps(\pi^*\xi, \pi^*\xi)= \eps^3 g_B(\xi, \xi)
\end{equation}
and $g_\eps - G_\eps$ vanishes on vertical forms. As to the potential $V_\rho$, we clearly have $\rho=1 + \mathcal{O}(\eps)$ and thus
\begin{align*}
 V_\rho&=\tfrac12 \Delta_{g_\eps}\log \rho + \tfrac14 g_\eps(\ud \log \rho, \ud \log \rho) + \mathcal{O}(\eps^3)\\
&=\begin{aligned}[t]
   &\tfrac12 \Delta_{F}\log \rho + \tfrac14 g_F(\ud \log \rho, \ud \log \rho)+\tfrac12\eps^2 \Delta_h \log \rho\\
&+\tfrac14 \eps^2(\pi^*g_B)(\ud \log \rho, \ud \log \rho) + \mathcal{O}(\eps^3)
  \end{aligned}\\
&=\tfrac12 \Delta_{F}\log \rho + \tfrac14 g_F(\ud \log \rho, \ud \log \rho)+ \mathcal{O}(\eps^3)\,.
\end{align*}
Here the second term is of order $\eps^2$, while the first is of order $\eps$ in general. 
Note however that for $\partial F=\varnothing$
\begin{equation*}
 \int_{F_x} \Delta_{F}\log \rho \vol_{g_{F_x}} =0\,.
\end{equation*}
Thus, in this case, we find using regular perturbation theory and the fact that the ground state eigenfunction of $\Delta_{F_x}$ is constant, that the ground state eigenvalue
\begin{equation*}
 \Lambda_\eps(x):=\min \sigma(-\Delta_F + V_\rho\vert_{F_x})=\mathrm{Vol}(F_x)^{-1}\int_{F_x} V_\rho \vol_{g_{F_x}} +  \mathcal{O}(\eps^2)
\end{equation*}
 is of order $\eps^2$.
 This is also true for typical examples with $\partial M \neq 0$, such as those discussed in~\cite{FrKr, HLT, KTnodal}, where actually $\Delta_F \log \rho=\mathcal{O}(\eps^2)$ holds point-wise. All of our results will be valid also in more general situations if the following condition holds.
\begin{cond*}
Let $\phi \in \ker(-\Delta_{F_x} - \Lambda_0)$ be normalised. Then
\begin{equation}\label{eq:cond rho}
 \sup_{x\in B}\int_{F_x} \abs{\phi}^2 V_\rho\,\mathrm{vol}_{g_{F_x}}=\mathcal{O}(\eps^2)\,. 
\end{equation}
\end{cond*}
\subsection{\textbf{Adiabatic perturbation theory}}\label{sect:adiab}
We now summarise the results of~\cite{LT} insofar as they are relevant to our specific situation. These results are concerned with the adiabatic operator $H_\mathrm{a}$, which is essentially given by the projection of $H$ to $\ker(\Delta_F-\Lambda_0)$.
This operator is still inherently $\eps$-dependent and $H_0$ will be obtained as its leading order contribution.  
To be more precise, let $\phi_0{:}\, M \to \R$ be the unique function such that for every $x\in B$ the restriction $\phi_0\vert_{F_x}$ is the positive and normalised ground state of $-\Delta_F$ (with Dirichlet boundary conditions if $\partial M\neq \varnothing$).
Let $\Lambda_\eps(x)$ be the smallest eigenvalue of $-\Delta_{F} + V_\rho\vert_{F_x}$ (with Dirichlet boundary conditions if $\partial M\neq \varnothing$) and $P_{\Lambda_\eps}(x)$ be the corresponding spectral projection.
Define $\phi_\eps \in L^2(M)$ by
\begin{equation*}
 \phi_\eps\vert_{F_x}:=P_{\Lambda_\eps}(x)\phi_0/\norm{P_{\Lambda_\eps}(x)\phi_0}_{L^2(F, g_{F_x})}\,.
\end{equation*}
The difference between $\phi_0$ and $\phi_\eps$ is of order $\eps$ in $L^2(M,g)$, as can be seen  immediately from the formula~\eqref{eq:phi eps}.
Since $g_F$ and $V_\rho$ are smooth, $\phi_0$ and $\phi_\eps$ are actually smooth functions not only on every fibre, but on $M$ (see~\cite[appendix B.2]{Lam} for a detailed proof).

The adiabatic operator is then given by 
\begin{equation*}
 \big(H_\mathrm{a}\psi\big)(x):= \langle\phi_\eps, H \phi_\eps \pi^*\psi \rangle_{L^2(F_x)} 
\end{equation*}
for $\psi\in D(H_\mathrm{a})= W^2(B)$. In the following we will often not make the pullback $\pi^*$ of $\psi$ explicit in the notation, that is we treat $\phi_\eps \psi$ as a function on $M$, even though $\psi$ is actually a function on $B$.
For the situation we wish to study here the relevant result can be formulated as follows:
\begin{thm}[\cite{LT}]\label{thm:adiab}
Let $\nu_0 \leq \nu_1 \leq \dots$ and $\lambda_0\leq \lambda_1 \leq \dots$ denote the eigenvalues of $H_\mathrm{a}$ and $H$ respectively, repeated  according to multiplicity. For every $J\in \N$ there exist constants $C_J$ and $\eps_0>0$ such that for all $j\leq J$ and $\eps < \eps_0$
\begin{equation*}
 \abs{\nu_j - \lambda_j}\leq C_J \eps^{4}\,.
\end{equation*}
 If in addition $\dist\big(\nu_j, \sigma(H_\mathrm{a})\setminus\lbrace \nu_j \rbrace\big) \geq C_j \eps^2$ for some $j\leq J$, then given a normalised eigenfunction $\psi_\eps\in \ker(H_\mathrm{a} - \nu_j)$ there is $\varphi\in \ker(H - \lambda_j)$ such that
\begin{equation*}
\norm{\phi_0\psi_\eps- \varphi}^2_{W^1(M,g)}=\int_M \abs{\phi_0\psi_\eps - \varphi}^2 + g\big(\ud (\phi_0\psi_\eps - \varphi), \ud (\phi_0\psi_\eps - \varphi)\big)\,\mathrm{vol}_g =\mathcal{O}(\eps^2)\,.
\end{equation*}
\end{thm}
Since $(-\Delta_F + V_\rho - \Lambda_0)\phi_\eps=\mathcal{O}(\eps^2)$ by condition~\eqref{eq:cond rho} we can isolate the leading order of this operator, that is we have (see~\cite[chapter 3]{Lam})
\begin{equation}\label{eq:Ha}
H_\mathrm{a}=-\eps^2\Delta_{g_B} + \eps^2 V_\mathrm{eff} + \mathcal{O}(\eps^3)=:\eps^2 H_0 + \mathcal{O}(\eps^3)\,,
\end{equation}
 with 
 \begin{equation*}
  V_\mathrm{eff}= V_\mathrm{a} + \eps^{-2} (\Lambda_\eps - \Lambda_0)\,,
 \end{equation*}
and the adiabatic potential 
\begin{equation*}
V_\mathrm{a}= \tfrac12 \tr_{g_B}\big(\nabla^B\bar\eta\big) - \int_{F_x} \pi^*g_B\big(\grad_{g} \phi_0, \grad_{g} \phi_0\big) \,\mathrm{vol}_{g_{F_x}}\,,
\end{equation*}
where $\nabla^B$ is the Levi-Cività connection of $g_B$ and $\bar \eta$ is the average of the mean curvature $\eta$ of the fibres,
\begin{equation*}
 \bar\eta(X):=\int_{F_x} \abs{\phi_0}^2 g_B(\pi_*\eta, X)\,\mathrm{vol}_{g_{F_x}}\,.
\end{equation*}
Equation~\eqref{eq:Ha} defines the operator $H_0$ and the remainder is an operator of order $\eps^3$ in the norm of $\mathscr{L}\big(W^2(B,g_B),L^2(B,g_B)\big)$.\label{H0-Ha} Hence by standard perturbation theory the eigenvalues $\nu$ of $H_\mathrm{a}$ are given by $\eps^2 \mu + \mathcal{O}(\eps^3)$, where $\mu$ is an eigenvalue of $H_0$. In particular $\mu$ is simple, if and only if $\nu$ is simple and separated from the rest of $\sigma(H_\mathrm{a})$ by a gap of order $\eps^2$, as required in the second part of theorem~\ref{thm:adiab}. 

Now if $\mu \in \sigma(H_0)$ is a simple eigenvalue and $\psi\in \ker(H_0 - \mu)$, we can easily construct $\psi_\eps\in \ker(H_\mathrm{a}-\nu)$ such that
\begin{equation*}
\norm{\psi_\eps - \psi}_{W^2(B,g_B)}=\mathcal{O}(\eps)\,,
\end{equation*}
which for $\dim B\leq 3$ implies 
\begin{equation*}
 \norm{\psi_\eps - \psi}_\infty=\mathcal{O}(\eps)\,.
\end{equation*}
Hence in order to prove convergence of $\varphi$ to $\psi \phi_0$ it will be sufficient to prove convergence to $\psi_\eps \phi_0$. The key idea to prove the latter is that this difference satisfies an elliptic boundary value problem on $M$. More precisely, note that both $\phi_0\psi_\eps$ and $\varphi$ are smooth functions on $M$ that vanish on the boundary and that additionally
\begin{equation*}
\big(H-\lambda\big)(\psi_\eps \phi_0 - \varphi)=\big(H-\lambda\big)\psi_\eps \phi_0\,.
\end{equation*}
Using that $\phi_0$ and $\psi_\eps$ are eigenfunctions of explicit elliptic operators on $F$ and $B$ respectively, one shows rather easily that the right hand side of the equation is small, not only in $L^2(M,g)$ but uniformly. Together with the $W^1$-estimate on the difference given by theorem~\ref{thm:adiab} this will imply smallness of $\norm{\phi_0\psi_\eps - \varphi}_\infty$ (see lemma~\ref{lem:phi eps}).
\begin{rem}
 If we use $\phi_0$ instead of $\phi_\eps$ in the definition of $H_\mathrm{a}$, i.e.~we set
\begin{equation*}
\big(\tilde H_\mathrm{a}\psi\big)(x):= \langle\phi_0, H \phi_0 \psi \rangle_{L^2(F_x)}\,,
\end{equation*}
we obtain the expression
\begin{equation}\label{eq:Ha-0}
 \tilde H_\mathrm{a}= -\eps^2\Delta_{g_B} + \eps^2 V_\mathrm{a} + \int_{F_x} \abs{\phi_0}^2 V_\rho \vol_{g_F} + \mathcal{O}(\eps^3)\,.
\end{equation}
This amounts to calculating $\Lambda_\eps - \Lambda_0$ in first order perturbation theory, so if $V_\rho=\mathcal{O}(\eps^2)$, then $H_\mathrm{a}= \tilde H_\mathrm{a} +\mathcal{O}(\eps^3)$.
On the other hand, if $V_\rho$ is only of order $\eps$, the second order of perturbation theory is of order $\eps^2$ and the operators $H_\mathrm{a}$ and $\tilde H_\mathrm{a}$ differ at leading order.

Note also that for the calculation of $V_\mathrm{a}$ it makes no difference whether we use $\phi_0$ or $\phi_\eps$ since 
$\phi_\eps - \phi_0=\mathcal{O}(\eps)$ as an element of $\mathscr{C}^1(B, L^2(F))$ by formula~\eqref{eq:phi eps} and~\cite[lemma 3.9]{LT}.
\end{rem}
\subsection{\textbf{A regularity lemma}}
Here we establish an elliptic regularity estimate for $-\Delta_{g_\eps}$ that takes into account the explicit $\eps$-dependence of the operator. In this, care needs to be taken since, written in a fixed system of local coordinates, the family $\lbrace-\Delta_{g_\eps}:0<\eps<1 \rbrace$ is not uniformly elliptic. For this reason we will choose $\eps$-dependent coordinate systems, basically $g_\eps$-geodesic coordinates, in which the local expressions for $H$ give a uniformly elliptic family of differential operators on some ball in $\R^m$.

The lemma will also be useful in other contexts, so we prove it in greater generality than required. In particular only for this section we will not assume $M$ to be compact.
Rather we assume that $\pi{:}\,M\to B$ is a fibre bundle of manifolds with boundary with compact fibre $F$, that the boundary of $B$ is empty and that $M$ carries a family of metrics of the form~\eqref{eq:geps} such that $(M, g_\eps)$ is of bounded geometry in the sense of Schick~\cite{Schi}, uniformly in $\eps$. This means that there exists $r>0$ such that for every $\eps$ there is an atlas $\mathfrak{U}^\eps:=\lbrace \kappa_j^\eps{:}\, U_j^\eps \to \R^m: j\in \mathbb{Z} \rbrace$ (denoting $m:=\dim M$)
of $M$ with the following properties:
\begin{itemize}
 \item For $j\geq 0$, $\kappa_j^\eps{:}\, U_j^\eps \to B(r, 0)$ is given by a system of $g_\eps$-geodesic coordinates centred at some point $x_j\in M$ with radius $r$.
\item For $j<0$, $\kappa_j^\eps{:}\, U_j^\eps \to B(r, 0)\times [0, r)$ is a boundary collar map, i.e~it extends a geodesic coordinate system $\beta$ on $(\partial M, g_\eps\vert_{\partial M})$ along the inward pointing normal $\nu$ of the boundary via $(\kappa_j^\eps)^{-1}(v,s)=\exp_{\beta^{-1}(v)}(s \nu)$.
\item The sets $(\kappa_j^\eps)^{-1}\big(B(2r/3 , 0)\big)$ for $j
\geq 0$ and $(\kappa_j^\eps)^{-1}\big(B(2r/3, 0)\times [0, 2r/3)\big)$ if $j<0$ form an open cover of $M$.
\item The coefficients of the metric tensor $(g_\eps)_{kl}$ and its dual $(g_\eps)^{kl}$ in these coordinate systems are bounded with all their derivatives, uniformly in $j$ and $\eps$.
\end{itemize}

For a compact manifold with the adiabatic family of metrics $g_\eps$ given by~\eqref{eq:geps} this is always satisfied because all of the quantities associated with $g_\eps$ that need to be bounded, such as curvatures and injectivity radii,  only become better as $\eps$ decreases (see~\cite[appendix A]{Lam}).

The proof of the lemma relies on a generalised maximum principle~\cite[theorem 10]{PW}, that was also used in earlier works~\cite{FrKr, GrJe, JerDiam, Jer} on this topic.
\begin{thm*}[The generalised maximum principle]
Let $\Omega\subset \mathbb{R}^k$ be a bounded domain. Let $D$ be a uniformly elliptic\footnote{Our convention is that $-\sigma(D)(\xi)\geq e\abs{\xi}^2$, $e>0$, where the symbol is defined by the relation $\sigma(D)(\ud f)=[[D,f],f]$.} 
operator of second order with coefficients in $\mathscr{C}^\infty(\overline\Omega)$. If $u, w\in \mathscr{C}^2(\Omega)\cap \mathscr{C}^0(\overline \Omega)$ satisfy the differential inequalities
\begin{align*}
&D u\geq 0\\
&D w\leq 0
\end{align*}
in $\Omega$ and $u>0$ in $\overline \Omega$ , then $w/u$ cannot attain a non-negative maximum in $\Omega$, unless it is constant. 
\end{thm*} 
Additionally we use the following well-known corollary (see e.g.~\cite[Corollary 3.14]{Lam} for a proof):
\begin{cor}\label{cor:positivity}
Let $\Omega= B(r,0)$ and let $D^0$ denote the operator $D$ with Dirichlet boundary conditions. Assume $D^0$ is self-adjoint and that ${w\in W^{1}(\Omega)\cap \mathscr{C}^0(\overline \Omega)}$ is strictly positive on $\partial \Omega$. Then if $\lambda<  \min \sigma (D^0)$, the unique solution $u \in \mathscr{C}^\infty(\Omega)\cap \mathscr{C}^0 (\overline \Omega)$ of the boundary value problem
\begin{align*}
D u  = \lambda u&\qquad \text{in}\quad \Omega\, ,\\
u=w  &\qquad\text{on}\quad \partial \Omega
\end{align*}
is strictly positive.
\end{cor}
\begin{lem}\label{lem:max}
Assume $(M, g_\eps)$ is a manifold of bounded geometry uniformly in $\eps$.
Let $\lambda(\eps)\geq 0$ with $\lim_{\eps \to 0} \lambda(\eps) =0$. If $f\in \mathscr{C}^2(M)$ is a solution of 
\begin{equation}\label{eq:Delta f}
 (H - \lambda(\eps))f=\delta\,,\qquad f\vert_{\partial M} =0\,.
\end{equation}
with $\norm{\delta}_\infty =\mathcal{O}(\eps)$, then there are positive constants $C$, $R$ and $\eps_0$ such that for every $x\in M$ and $\eps< \eps_0$
\begin{equation}\label{eq:f bound}
\abs{f(x)}\leq C \Bigg(\bigg( \int\limits_{\Omega(x)} \abs{f}^2 + g_\eps(\ud f, \ud f) \,\vol_{g_\eps} \bigg)^{1/2} +  \eps\Bigg)\,,
\end{equation}
where $\Omega(x):=\lbrace y\in M: \dist_{g_B}(\pi(y), \pi(x))<\eps R\rbrace$.
\end{lem}
\begin{proof}
 We prove the statement for the positive part $f_+$ of $f$, the proof for the negative part is identical. 
Let $K>0$ be a constant such that $\norm{\delta}_\infty + \norm{V_\rho}_\infty\leq K\eps$ for $\eps$ small enough and note that in the interior of $\Omega_+:=\supp f_+$ we have
\begin{equation}\label{eq:Hdelta ineq}
\left(H -\lambda(\eps) - K\right)\left(f_+ + \eps\right)= \underbrace{\delta - \eps K + \eps V_\rho}_{\leq 0} - K f_+ - \eps(\lambda(\eps)+ \Lambda_0) \leq 0\,.
\end{equation}
We now aim at constructing a function $u$, defined on a neighbourhood of $x$, with $u\geq f_+$, but bounded by the integral in the statement of the lemma.
This will be achieved by choosing $u$ as the solution of an elliptic boundary value problem and then using the maximum principle.

We will now work in the atlas $\mathfrak{U}^\eps$ introduced in the introduction to this section. The uniform estimates on geometric quantities expressed in these coordinates will make our locally obtained estimates hold uniformly on $M$.
The virtue of these $\eps$-dependent coordinate systems is that they mitigate (the leading order of) the $\eps$-dependence of $g_\eps$ since in geodesic coordinates this leading order is always given by the Euclidean metric. Since the bounded geometry of $(M, g_\eps)$ is uniform in $\eps$, we have uniform bounds on the expression $g_\eps^{kl}- \delta^{kl}$ and its derivatives in these coordinate systems. Moreover there exists an $\eps$-independent radius $r_0\leq r$ such that for every $x\in M$ and $0<\eps<1$ there is $j\in \mathbb{Z}$ for which the Euclidean ball $B\big(r_0, \kappa_j^\eps(x)\big)\subset \R^m$ (or $B\big(r_0, \kappa_j^\eps(x)\big)\cap \lbrace x_{m}\geq 0\rbrace$ if  $j<0$) is completely contained in the image of $\kappa_j^\eps$.

Now fix $x\in \Omega_+$, and let $\kappa^\eps{:}\,U^\eps \to \R^m$ be the coordinate system with the property above, shifted so that $\kappa^\eps(x)=0$. Set $D_x^\eps:=\kappa_*^\eps H$ and if $\kappa^\eps$ is a boundary chart extend this to an elliptic operator on $B(r_0,0)$ by smoothly extending its coefficients. This yields a family of elliptic operators that have common bounds $\lbrace e, c\rbrace$, independent of $x$ and $\eps$, on their ellipticity constants and coefficients.
Hence there exists a positive radius $r_1\leq r_0$ such that for all $x$ and $\eps$ we have a lower bound on the Dirichlet energy:
\begin{equation*}
\inf_{0\neq \phi \in W^{1}_0\big(B(r_1,0)\big)} 
\frac{\langle\phi, D_x^\eps \phi\rangle_{L^2(B(r,0))}}{\langle \phi, \phi \rangle} \geq 3K \,.
\end{equation*}
Now $w=\kappa^\eps_*f_+$ (extended by zero if $\kappa^\eps$ is a boundary chart) defines a function in $W^1\big(B(r_1, 0)\big)\cap \mathscr{C}^0\big(\overline{B(r_1, 0)}\big)$ and the boundary value problem
\begin{equation}\label{eq:def u}
 D_x^\eps u = 2 K u\,,\qquad	u\vert_{\partial B(r_1, 0)}=w + \eps
\end{equation}
has a unique solution $u\in \mathscr{C}^\infty(B(r_1,0))\cap \mathscr{C}^0\big(\overline{B(r_1, 0)}\big)$(see~\cite[chapter 8]{GT}), which is strictly positive by corollary~\ref{cor:positivity}.
Positivity of $u$ entails that $(D_x^\eps -K)u=Ku>0$, while $(D_x^\eps -K)\eps=\kappa_*^\eps V_\rho\eps -\Lambda_0\eps  - K\eps \leq 0$. Hence by the maximum principle
\begin{equation}\label{eq:u lb}
0< \sup_{y\in B(r_1, 0)} \frac{\eps}{u} = \max_{y\in \partial B(r_1, 0)}\frac{\eps}{u} \leq 1\,.
\end{equation}
Now for $\eps$ small enough we have $\lambda(\eps) <K$ and hence
\begin{equation*}
(D_x^\eps - \lambda - K)u=(K - \lambda)u>0
\end{equation*}
in $B(r_1, 0)$, while $w$ satisfies
\begin{equation*}
(D_x^\eps - \lambda - K)(w+\eps)\stackrel{\eqref{eq:Hdelta ineq}}{\leq} 0  
\end{equation*}
in $\kappa^\eps(\Omega_+ \cap U^\eps)$. Thus by the maximum principle the function $(w+\eps)/u$, defined on $\kappa^\eps(\Omega_+ \cap U^\eps)\cap B(r_1, 0)$ attains its maximum on the boundary of this set. On the boundary of $B(r_1, 0)$ the quotient equals one by~\eqref{eq:def u}, while on the boundary of $\kappa^\eps(\Omega_+)$ we have $w\equiv 0$ and $(w + \eps)/u =\eps/u\leq 1$ by~\eqref{eq:u lb}.
Consequently
\begin{equation*}
(w+\eps)/u\leq 1
\end{equation*}
and in particular
\begin{equation*}
 f_+(x)=w(0)\leq u(0)\,.
\end{equation*}
The rest of the proof consists in bounding $u(0)$ by the right hand side of~\eqref{eq:f bound}. To start with, we have the a priori estimate~\cite[corollary 8.7]{GT} (with $C= C(K, r_1, e, c)$) 
\begin{equation*}
\norm{u}_{W^{1}(B(r_1,0))} \leq C \big(\lVert w \rVert_{W^{1}(B(r_1,0))}+\eps \big)\,.
\end{equation*}
For higher Sobolev norms interior elliptic regularity~\cite[theorem 8.10]{GT} gives
\begin{equation*}
\norm{u}_{W^{k}(B(r_1/2,0))}\leq C(k, r_1, e , c, K) \norm{u}_{W^{1}(B(r_1,0))}\,,
\end{equation*}
since $u$ is an eigenfunction of $D_x^\eps$.
If we take $k>(m +1) /4$ the Sobolev embedding theorem gives a bound on $\sup_{y\in B(r_1/2,0)} u(y)$ and in particular on $u(0)$.
Hence we have
\begin{equation*}
u(0)\leq C(m, r_1, e, c, K)\big(\norm{w}_{W^{1}(B(r_1,0))}  + \eps\big)\,.
\end{equation*}
Choose $R$ such that for every $x$ and every $0<\eps<1$, $(\kappa^\eps)^{-1}\big(B(r_1, \kappa^\eps(x)\big)$ is contained in the metric ball $\lbrace y \in M: \dist_{g_\eps}(x,y)< R\rbrace$.
Then there is a constant $C$, depending on the constants bounding the geometry of $(M, g_\eps)$, such that
\begin{equation*}
 \norm{w}_{W^{1}(B(r_1,0))} \leq C \norm{\kappa^{\eps*}w}_{W^{1}(\kappa^{-1}(B), g_\eps)}\,,
\end{equation*}
and we obtain
\begin{equation*}
f_+(x)\leq C  \Bigg(\bigg( \int\limits_{\Omega(x)} \abs{f}^2 + g_\eps(\ud f, \ud f) \,\vol_{g_\eps} \bigg)^{1/2} +  \eps\Bigg)\,,
\end{equation*}
which completes the proof.
\end{proof}
For a solution to an elliptic equation such as~\eqref{eq:Delta f} we naturally obtain estimates on derivatives. Since the family of operators $\Delta_{g_\eps}$, which dominates the behaviour of $H$, is associated with the metric $g_\eps$ such estimates hold for the function $f$ of lemma~\ref{lem:max} in the $\mathscr{C}^k$-norm corresponding to this metric. More precisely let for $f\in \mathscr{C}^k(M)$
\begin{equation*}
 \norm{f}_{\mathscr{C}^k(M, g_\eps)}:= \norm{f}_\infty + \sum_{j=1}^k \max_{x\in M} \sqrt{g_\eps(\nabla^j f, \nabla^j f)}\,,
\end{equation*}
where $\nabla^1=\ud$ denotes exterior derivation, higher derivatives are induced by the Levi-Cività connection of $g_\eps$ and the metric is canonically extended to the tensor bundles.
\begin{lem}\label{lem:reg}
Let $f\in \mathscr{C}^{2}(M)$ satisfy equation~\eqref{eq:Delta f}, then there exists a constant $C$ independent of $\eps$ such that
\begin{equation*}
\norm{f}_{\mathscr{C}^{2}(M, g_\eps)}\leq C\big( \norm{f}_\infty + \norm{\delta}_{\infty}\big)\,.
\end{equation*}
\end{lem}
\begin{proof}
Apply the local version of the statement~\cite[lemma 6.4]{GT} in the atlas $\mathfrak{U}^\eps$.
\end{proof}
\subsection{\textbf{Uniform convergence of eigenfunctions}}
We are now ready to prove uniform convergence of eigenfunctions of $H$ to $\psi\phi_0$, where $\psi$ is an eigenfunction of $H_0$, under the appropriate conditions. We remind ourselves that here $M$ is compact again and $\phi_0 \in \mathscr{C}^\infty(M)$ is independent of $\eps$, so for any $k\in \N$, $\norm{\phi_0}_{\mathscr{C}^k(M,g)}\leq C(k)$. Since $\phi_\eps$ is a ground state of an operator whose coefficients are smooth and bounded independently of $\eps$ we also have $\norm{\phi_\eps}_{\mathscr{C}^k(M,g)}\leq C(k)$ for an appropriate choice of constants (c.f~\cite[appendix B]{Lam}). In fact since $\phi_\eps$ is the ground state of a smooth perturbation of $-\Delta_{g_F}$ the difference of $\phi_\eps$ and $\phi_0$ is of order $\eps$ in $\mathscr{C}^\infty(M,g)$. However we only make use of this estimate in the norm of $\mathscr{C}^0$, for which we give a simple proof here.
\begin{lem}\label{lem:phi eps}
Let $\phi_0$ and $\phi_\eps$ be the functions on $M$ defined in section~\ref{sect:adiab}. Then 
\begin{equation*}
 \norm{\phi_0 - \phi_\eps}_\infty=\mathcal{O}(\eps)
\end{equation*}
\end{lem}
\begin{proof}
It is clearly sufficient to prove that $(1-P_{\Lambda_\eps})\phi_0 =\mathcal{O}(\eps)$.
Since the spectrum of $-\Delta_{F_x}$ depends continuously on $x \in B$ (by the min-max principle), we have
\begin{equation*}
 \inf_{x\in B} \big( \sigma(-\Delta_F\vert_{F_x})\setminus \lbrace \Lambda_0 \rbrace \big ) - \Lambda_0 
= \min_{x\in B} \big( \sigma(-\Delta_F\vert_{F_x})\setminus \lbrace \Lambda_0 \rbrace \big ) - \Lambda_0
= c>0\,.
\end{equation*}
For $\eps$ small enough we then have the integral representation
\begin{equation}\label{eq:phi eps}
 (1-P_{\Lambda_\eps})\phi_0\big\vert_{F_x}= \frac{\ui}{2\pi} \int\limits_{\abs{z-\Lambda_0}=c/2} \frac{1}{-\Delta_F + V_\rho -z} V_\rho \frac{1}{-\Delta_F -z} \phi_0\bigg\vert_{F_x}\ud z
\end{equation}
which follows from the Riesz formula for the spectral projection.
Now since for any $k\in \N$: ${\norm{V_\rho}_{\mathscr{C}^k(M,g)}=\mathcal{O}(\eps)}$, the integrand is an operator of order $\eps$ on $W^k(F_x, g_{F_x})$. 

Now let $r>0$ be less than the injectivity radius of $(B,g_B)$ and choose points $\lbrace x_i:i\in  I \rbrace$ such that the metric balls $B(r/2, x_i)$ form a finite open cover of $B$. Then let $U_i:=B(r, x_i)$ be the open cover by balls of double radius and choose trivialisations $\Phi_i{:}\, \pi^{-1}(U_i) \to U_i \times F$. The map $\Phi_i$ has bounded derivatives on $\pi^{-1}(B(r/2, x_i))$ (with respect to some fixed metric $g_0$ on $F$), so for $x\in B(r/2, x_i)$ it induces a bounded map $\Phi_{i*}{:}\,W^k(F_x, g_{F_x})\to W^k(F, g_0)$ with norm less than a constant $C(\Phi_i)$.
Then taking $k>\dim F/2$ the Sobolev embedding theorem, applied to $\Phi_{i*}(1-P_{\Lambda_\eps})\phi_0$, gives
\begin{equation*}
 \norm{(1-P_{\Lambda_\eps})\phi_0}_\infty
 =\max_{i\in I}  \norm{\Phi_{i*}(1-P_{\Lambda_\eps})\phi_0}_\infty
 \leq C \eps \max_{i\in I} C(\Phi_i)
 =\mathcal{O}(\eps)\,.
\end{equation*}
\end{proof}
\begin{thm}\label{thm:uniform}
Assume $d:=\dim B\leq 3$.
Let $\mu$ be a simple eigenvalue of $H_0$ and ${\psi\in \ker(H_0 - \mu)}$ a normalised eigenfunction. Then there exists $\eps_0>0$ such that for every $0<\eps<\eps_0$ there exists a simple eigenvalue ${\lambda=\eps^2\mu +\mathcal{O}(\eps^3)}$ of $H$ and a normalised eigenfunction $\varphi \in \ker(H - \lambda)$ such that $\varphi$ converges uniformly to $\psi\phi_0$ as $\eps\to 0$. 
More precisely there exists a constant $C$ such that for all $\eps< \eps_0$
\begin{equation*}
 \norm{\varphi - \psi\phi_0}_\infty \leq C \eps \theta_d(\eps)\,,
\end{equation*}
where $\theta_1(\eps)\equiv 1$, $\theta_2(\eps)= \sqrt{\log{\eps^{-1}}}$, $\theta_3(\eps)=\eps^{-1/2}$.
Additionally, if $\psi$ is real then $\varphi$ may be chosen real.
\end{thm}
\begin{proof}
 Perturbation theory for $H_\mathrm{a}=\eps^2 H_0 + \mathcal{O}(\eps^3)$ gives us an eigenvalue $\nu=\eps^2\mu + \mathcal{O}(\eps^3)$ and an eigenfunction $\psi_\eps\in \ker(H_\mathrm{a} -\nu)$ with $\norm{\psi - \psi_\eps}_{\infty} =\mathcal{O}(\eps)$. Theorem~\ref{thm:adiab} gives existence of $\lambda$ and $\varphi$. Note that if $\psi$ is real we may choose both $\psi_\eps$ and $\varphi$ real, since their imaginary parts are necessarily small.
 
 It remains to prove that $\norm{\psi_\eps \phi_0 - \varphi}_{\infty}$ converges to zero, or in view of lemma~\ref{lem:phi eps} that $\norm{\psi_\eps \phi_\eps - \varphi}_{\infty} \to 0$.
 This will be achieved by using lemma~\ref{lem:max} with $f:=\psi_\eps\phi_\eps - \varphi$ and $\lambda(\eps)=\lambda$. We have
 \begin{equation*}
  (H-\lambda)f= (H-\lambda)\psi_\eps\phi_\eps\,,
 \end{equation*}
 which is exactly of the form~\eqref{eq:Delta f} with
\begin{align*}
 \delta:&= (H-\lambda)\psi_\eps\phi_\eps
=\big(-\eps^2\Delta_h +\eps^3 S_\eps- \lambda\big)\psi_\eps\phi_\eps +  \psi_\eps\big(-\Delta_F + V_\rho - \Lambda_0\big)\phi_\eps\\
&=\big(-\eps^2\Delta_h +\eps^3 S_\eps- \lambda + \Lambda_\eps - \Lambda_0\big)\psi_\eps\phi_\eps\,.
\end{align*}
Since $\lambda=\mathcal{O}(\eps^2)=\Lambda_\eps - \Lambda_0$ by condition~\eqref{eq:cond rho}, this is bounded by 
\begin{equation*}
C\eps^2 \norm{\phi_\eps \psi_\eps}_{\mathscr{C}^2(M, g)}\leq 
C\eps^2 \norm{\phi_\eps}_{\mathscr{C}^2(M,g)} \norm{\psi_\eps}_{\mathscr{C}^2(B, g_B)}
=\mathcal{O}(\eps^2)\,,
\end{equation*}
because $\psi_\eps$ is a bounded eigenfunction of $H_\mathrm{a}$.

Now lemma~\ref{lem:max} may be applied and we need to estimate the integral on the right hand side of equation~\eqref{eq:f bound}
for our choice of $f=\psi_\eps\phi_\eps - \varphi$. 
To begin with we write
\begin{equation*}
\int\limits_{\Omega(x)} \abs{f}^2 + g_\eps(\ud f, \ud f) \,\vol_{g_\eps}
= \int\limits_{\Omega(x)} (1+\Lambda_0 - V_\rho)\abs{f}^2 + g_\eps(\ud f, \ud f) + (V_\rho- \Lambda_0)\abs{f}^2\,\vol_{g_\eps}\,.
\end{equation*}
Since, by condition~\eqref{eq:cond rho}, $-\Delta_F + V_\rho -\Lambda_0 \geq - c\eps^2$ we can estimate the latter terms by integrating over the whole of $M$.
That is, using that $\vol_{g_\eps}=\eps^{-d}\vol_{g}$, the latter terms are bounded by
\begin{align}
 \int\limits_M& g_\eps(\ud f, \ud f) +  (V_\rho- \Lambda_0 + c\eps^2 )\abs{f}^2\,\vol_{g_\eps}\notag\\
&\stackrel{\hphantom{(4)}}{=} \eps^{-d} \langle f, (-\Delta_{g_\eps}  + V_\rho - \Lambda_0 + c\eps^2) f\rangle_{L^2(M,g)}\notag\\
&\stackrel{\eqref{eq:Delta f}}{=}\eps^{-d} \langle f, \delta \rangle\notag 
+ \eps^{-d} \langle f, (\lambda + c \eps^2- \eps^3 S_\eps) f\rangle 
\end{align}
Now the fact that $(H-\lambda)f=(H-\lambda)\phi_\eps \psi_\eps$ and $H$ is self-adjoint implies
\begin{equation*}
\eps^{-d}\langle f , \delta \rangle_{L^2(M,g)} = \eps^{-d}\langle \psi_\eps, (H_\mathrm{a} - \lambda) \psi_\eps \rangle_{L^2(B,g_B)}=\eps^{-d}(\nu - \lambda)\stackrel{\ref{thm:adiab}}{=}\mathcal{O}(\eps^{4-d})\,.
\end{equation*}
The second term is easily estimated by
\begin{equation*}
\eps^{-d} \abs{\langle f, (\lambda + c \eps^2- \eps^3 S_\eps) f\rangle} \leq C\eps^{2-d} \norm{f}^2_{W^1(M,g)} \stackrel{\ref{thm:adiab}}{=} \mathcal{O}(\eps^{4-d})\,,
\end{equation*}
since $\norm{\phi_0 -\phi_\eps}^2_{W^1(M,g)}=\mathcal{O}(\eps^2)$ by equation~\eqref{eq:phi eps} and~\cite[lemma 3.8]{LT}.

To estimate the integral of $\abs{f}^2$ over $\Omega(x)$ we will resort to local arguments.
Let $r>0$ and $\lbrace (U_i, \Phi_i): i \in I \rbrace$ be the cover of $B$ and trivialisations chosen in the proof of lemma~\ref{lem:phi eps}. This has the property that for every $x\in B$ there is $i(x)\in I$ such that $B(r/2, x)\subset U_{i(x)}$.
Then if $R$ is the radius we get from lemma~\ref{lem:max}, $\Omega(x)$ is contained in $\pi^{-1}(U_{i(x)})$ for $\eps R\leq r/2$. 
We may thus rewrite the integral over $\Omega(x)$ in the trivialisation $\Phi:=\Phi_{i(x)}$
\begin{equation*}
 \eps^{-d}\int\limits_{\Omega(x)} \abs{f}^2 \vol_g = \eps^{-d}\int\limits_{ F}\int\limits_{\pi(\Omega)} \abs{\Phi_*f}^2 \vol_{g_B}\vol_{\Phi_*g_F}\,.
\end{equation*}
Let $\xi \in \mathscr {C}^\infty_0(\R^d)$ be zero outside of $B(r/2, 0)$ and equal to one on $B(r/4,0)$ and let $\chi \in \mathscr {C}^\infty_0(B(r/2,x))$ be the pullback of $\xi$ by a geodesic coordinate system centred at $x$.
Then we claim that
\begin{equation*}
\eps^{-d} \int\limits_{\pi(\Omega)} \abs{\Phi_*f}^2 \vol_{g_B} \leq C \theta_d^2(\eps) \int\limits_{B(r/2,x)} \abs{\chi \Phi_* f}^2 + g_B(\ud \chi \Phi_* f, \ud \chi \Phi_* f) \vol_{g_B}
\end{equation*}
for $\eps R \leq r/4$. In fact, since $\mathrm{Vol}_{g_B}(\pi(\Omega))=\mathcal{O}(\eps^d)$, for $d=1$ this is just the Sobolev embedding theorem and for $d=2,3$ it can be shown in a similar way, e.g~by using the Fourier transform in local coordinates (see~\cite[appendix C]{Lam}). Integrating this inequality over $F$ gives
\begin{equation*}
\eps^{-d}\int\limits_{\Omega(x)} \abs{f}^2 \vol_g \leq C \theta_d^2(\eps) \norm{f}^2_{W^1(M,g)}\,, 
\end{equation*}
with a constant $C$ that is independent of $x$. Since $\norm{f}^2_{W^1(M,g)}=\mathcal{O}(\eps^2)$ this completes the proof.
\end{proof}
\begin{cor}\label{cor:reg}
Let $M$, $\psi$, $\varphi$ and $\theta_d(\eps)$ be as in theorem~\ref{thm:uniform}. Then there exist constants $C$ and $\eps_0>0$ such that for every $\eps<\eps_0$
\begin{equation*}
 \norm{\psi\phi_0 - \varphi}_{\mathscr{C}^2(M, g_\eps)}\leq C \eps \theta_d(\eps)\,.
\end{equation*}
\end{cor}
\begin{proof}
We have
\begin{align}
\big(H-\lambda\big)(\psi\phi_0 - \varphi)&=\big(H-\lambda\big)\psi\phi_0\notag\\
&=\big(-\eps^2\Delta_h +\eps^3 S_\eps + V_\rho - \lambda\big)\psi\phi_0 + \psi \underbrace{\big(-\Delta_F - \Lambda_0)\phi_0}_{=0}\label{eq:Hpsi}
\end{align}
 and this is bounded by $C\eps \norm{\phi_0}_{\mathscr{C}^2(M,g)} \norm{\psi}_{\mathscr{C}^2(B, g_B)}
=\mathcal{O}(\eps)$. Hence lemma~\ref{lem:reg} together with theorem~\ref{thm:uniform} proves the claim.
\end{proof}
\subsection{\textbf{Closed manifolds fibred over the circle}}\label{sect:S1}
Now that we have uniform convergence of the appropriate eigenfunctions we will study the nodal set of these functions in the case $\partial M =\varnothing$ and $B=S^1$. The estimation of the location of $\mathcal{N}(\varphi)$ will be rather simple to derive in this context. This allows us to exhibit the structure of the argument in a clear way, which will prove useful for the proof in the slightly more involved case of section~\ref{sect:node gen}.
In addition, in this specific case we are able to obtain significantly stronger results. Namely we prove in theorem~\ref{thm:iso} that the nodal set $\mathcal{N}(\varphi)$ is the disjoint union of embedded submanifolds $\iota_k{:}\,F\to M$, one for every element of $\psi^{-1}(0)$, and every one of these submanifolds is isotopic to a fibre of $\pi{:}\,M\to B$.   

Throughout this section we will work in a fixed model of $B$, namely  we let ${L:=\mathrm{diam}(B)}$ and make use of the isometry of $(B, g_B)$ with $\big(\R/2L\mathbb{Z}, \ud s^2\big)$, where we denote points by $s$. Since also $\partial M=\varnothing$ we have $\Lambda_0=0$ and $\phi_0=\pi^* \mathrm{Vol}(F_s)^{-1/2}$.
Additionally, using $\sqrt \rho$ as a trial function, one easily checks that $\Lambda_\eps=\mathcal{O}(\eps^3)$. 
Plugging this into equation~\eqref{eq:Ha} we find that the relevant operator has the simple form
\begin{equation}\label{eq:H0 circ}
 H_0=-\partial_s^2 + \tfrac12 \partial_s^2\log \mathrm{Vol}(F_s) + \tfrac14 \abs{\partial_s \log(\mathrm{Vol}(F_s))}^2
\end{equation}
as an operator on the interval $(-L, L)$ with periodic boundary conditions. Theorem~\ref{thm:uniform} applies to every simple eigenvalue of $H_0$, and generically, that is if $\mathrm{Vol}(F_s)$ is an arbitrary $2L$-periodic function of $s$, all of the eigenvalues of $H_0$ are simple. This allows us to estimate the location of $\mathcal{N}(\varphi)$ based on the analysis of the behaviour of $\psi$.
\begin{prop}\label{prop:node}
Assume $\partial M =\varnothing$ and $B\cong S^1$.
Let $\mu$ be a simple eigenvalue of $H_0$ with real, normalised eigenfunction $\psi$ and $\varphi\in \ker(H-\lambda)$ the corresponding eigenfunction of $H$ provided by theorem~\ref{thm:uniform}. There exist constants $C>0$ and $\eps_0>0$, such that for every $\eps<\eps_0$ and $y \in M$ we have  
\begin{equation*}
 \dist_{g_B}\big(\pi(y),\mathcal{N}(\psi)\big)\geq C\eps \implies 
\mathrm{sign}\big(\varphi(y)\big)=\mathrm{sign}\big(\psi(\pi(y))\big)\,.
\end{equation*}
\end{prop}
\begin{proof}
We carry out the proof in several small steps that will reappear in the proof of the more general statement of theorem~\ref{thm:node}.

Let $\mathcal{N}(\psi)=\lbrace s_i: i\in I\rbrace$ be the (finite) set of zeros of $\psi$.
\begin{enumerate}[label=\arabic*)]
 \item\label{step1} There is $C_0>0$ such that $\abs{\partial_s\psi}(s_i)\geq C_0$ for every $i\in I$:
 
 $\psi$ solves the second order ordinary differential equation
\begin{equation*}
\partial_s^2 \psi= \left( V_\eff- \mu \right) \psi\,,
\end{equation*}
so if at any point $\psi(s)=\partial_s\psi(s)=0$ it must vanish everywhere since zero is the unique solution of the equation with that initial condition. Thus the derivative of $\psi$ cannot vanish at any $s_i$. It is also independent of $\eps$, since $\psi$ does not depend on $\eps$ at all. $C_0$ may now be chosen as the minimum of $\abs{\partial_s \psi(s_i)}$ over the finite set $I$.
\item\label{step2} For any $C_1>0$ and $\eps$ small enough, we have $\abs{\psi(s_i\pm C_2 \eps)}>C_1\eps$ with $C_2:=2 C_1/C_0$:

By Taylor expansion
\begin{equation*}
 \abs{\psi(s_i + 2C_2/C_0\eps)}=2C_1 \eps \abs{\partial_s \psi}(s_i)/C_0 + \mathcal{O}(\eps^2)> C_1 \eps\,.
\end{equation*}
\item\label{step3} If $\mathrm{dist}(s,\mathcal{N}(\psi))\geq C_2 \eps$, then $\abs{\psi(s)}>C_1 \eps$ for $\eps$ small enough:

If in the interval $[s_i, s_j]$ between two consecutive zeros of $\psi$ there is no local minimum of $\abs{\psi}$, then $\abs{\psi}$ attains its minimum on the boundary of $[s_i + C_2\eps, s_j - C_2\eps]$, where it is larger than $C_1\eps$ by step~\ref{step2}.

If on the other hand there is a local minimum at $s^*\in[s_i, s_j]$ we just need $\eps$ to be small enough to ensure that
$\abs{\psi}(s^*)>C_1\eps$.
\item\label{step4} Denote by $C_3$ the constant of theorem~\ref{thm:uniform} with $d=1$.
The proof is completed by letting $C_1:=\lVert\phi_0^{-1}\rVert_\infty C_3$ and $C:=C_2=2C_1/C_0$:

First note that since $\partial F=\varnothing$ we have $\phi_0=\pi^*\mathrm{Vol}(F_x)^{-1/2}$, so $\lVert\phi_0^{-1}\rVert_\infty$ is finite.
Now let $y\in M$ with $x=\pi(y)$ satisfy $\dist_{g_B}\left(x,\mathcal{N}(\psi)\right)\geq C \eps$. By step~\ref{step3} we have 
\begin{equation*}
\abs{\psi \phi_0(y)}
>C_3 \eps \phi_0 \norm{\phi_0^{-1}}_\infty\geq C_3 \eps\,,
\end{equation*}
and since by theorem~\ref{thm:uniform}
\begin{equation*}
 \Vert\psi\phi_0 - \varphi\Vert_\infty
 \leq C_3\eps\,,
\end{equation*}
$\varphi$ must have the same sign as $\psi$.
\end{enumerate}
\end{proof}
The estimate on $\mathrm{sign}(\varphi)$ we just derived tells us that $\varphi$ is non-zero far from $\pi^{-1}\mathcal{N}(\psi)$, and also that it must change sign in a neighbourhood of this set. This entails convergence of the nodal set rather directly.
\begin{cor}\label{cor:Hausdorff}
Let $\psi$, $\varphi$, $C$ and $\eps_0$ be as in proposition~\ref{prop:node}. Then for all $\eps<\eps_0$ 
\begin{equation*}
\dist\big(\mathcal{N}(\varphi), \pi^{-1}\mathcal{N}(\psi)\big)\leq C\eps\,,
\end{equation*}
where $\dist$ denotes the Hausdorff distance with respect to the metric $g$.
\end{cor}
\begin{proof}
For $\delta\geq 0$ and a compact set $K\subset M$ denote by $T_\delta(K)$ the closed $\delta$-tube
\begin{equation*}
 T_\delta(K)=\lbrace x\in M: \dist_{g}(x, K)\leq \delta \rbrace\,.
\end{equation*}
The Hausdorff distance is given by
\begin{equation*}
 \dist(K, \tilde K)=\inf \lbrace \delta \geq 0: K\subset T_\delta(\tilde K)\text{ and }\tilde K\subset T_\delta(K)\rbrace\,.
\end{equation*}
Proposition~\ref{prop:node} shows that $\mathcal{N}(\varphi)\subset T_{C\eps}(\pi^{-1}\mathcal{N}(\psi))$. 
Now let $s_i\in \mathcal{N}(\psi)$, in the notation of proposition~\ref{prop:node}, and let $v_+\in\lbrace\pm \partial_s\rbrace$ denote the normalised tangent vector at $s_i$ pointing into the region where $\psi$ is positive. Then for $x\in F_{s_i}$ let $v_+^*\in (TF)^\perp_x$ be the unique horizontal vector with $\pi_* v_+^*=v_+$ and define $\gamma$ to be the $g$-geodesic 
\begin{equation*}
\gamma{:}\,[-C\eps, C \eps]\to M\qquad t\mapsto \exp_x(t v_+^*)\,.
\end{equation*}
This horizontal geodesic projects to the $g_B$-geodesic $t \mapsto s_i \pm t$ so by proposition~\ref{prop:node} we know that $\varphi(\gamma(C\eps))>0$ and $\varphi(\gamma(-C\eps))<0$. Consequently $\gamma\cap \mathcal{N}(\varphi)\neq \varnothing$ and since $\gamma(0)=x$ this proves $\dist(x, \mathcal{N}(\varphi))\leq C\eps$. Because $x$ was arbitrary this shows $\pi^{-1}\mathcal{N}(\psi)\subset T_{C\eps}(\mathcal{N}(\varphi))$ and the statement of the corollary.
\end{proof}
In order to determine the homotopy class of $\mathcal{N}(\varphi)$ we will show that the point $t_0\in (-C\eps, C\eps)$ where the curve $\gamma$ defined above cuts the nodal set is a well-defined, smooth function of $x\in F$. An isotopy is then given by moving along these curves. In order to achieve this we must show that $\dot\gamma\varphi$ does not vanish close to the nodal set. 
Since we know that $\dot\gamma\pi^*\psi=\partial_s \psi=\mathcal{O}(1)$ in this region and we believe that $\phi_0\psi$ is a good approximation of $\varphi$ this should be true. However the estimate $\norm{\phi_0\psi - \varphi}_{\mathscr{C}^1(M, g_\eps)}=\mathcal{O}(\eps)$ of corollary~\ref{cor:reg}  
only gives $\norm{\dot\gamma(\psi\phi_0 -\varphi)}_\infty= \mathcal{O}(1)$ because $\dot\gamma$ has length $\eps^{-1}$. This is useless for the task at hand, so we need to prove a refined estimate on the difference of the horizontal derivatives.
\begin{lem}\label{lem:deriv}
Let $M$, $\psi$ and $\varphi$ be as in proposition~\ref{prop:node} and let $\partial_s^*$ denote the unique lift of $\partial_s$ to a horizontal vector field on $M$.
Then there exist $C$ and $\eps_0>0$ such that
\begin{equation*}
\norm{\partial_s^*(\psi\phi_0 -\varphi)}_\infty \leq C \sqrt\eps
\end{equation*}
for all $\eps<\eps_0$.
\end{lem}
\begin{proof}
 The proof of this statement proceeds by applying lemma~\ref{lem:max} with
\begin{equation*}
 f:=\partial_s^*(\psi\phi_0 -\varphi)
\end{equation*}
and $\lambda(\eps):=\lambda$.
Hence we calculate
\begin{equation*}
 (H-\lambda)f= [H, \partial_s^*](\psi\phi_0 -\varphi) +\partial_s^*(H-\lambda)\psi\phi_0=:\tilde \delta\,.
\end{equation*}
$[H, \partial_s^*]$ is a second order differential operator in which every horizontal derivative is accompanied by a factor $\eps$, so using corollary~\ref{cor:reg} we have
\begin{equation*}
 \norm{[H, \partial_s^*](\psi\phi_0 -\varphi)}_\infty \leq C\norm{\psi\phi_0 - \varphi}_{\mathscr{C}^2(M, g_\eps)}=\mathcal{O}(\eps)\,.
\end{equation*}
By equation~\eqref{eq:Hpsi} we have
\begin{equation*}
\norm{\partial_s^*(H - \lambda)\psi \phi_0}_\infty \leq C \eps \norm{\phi_0 \psi}_{\mathscr{C}^3(M, g)}=\mathcal{O}(\eps)
\end{equation*}
and hence $\tilde \delta=\mathcal{O}(\eps)$. 

After applying lemma~\ref{lem:max} we must estimate the integral on the right hand side of~\eqref{eq:f bound} to complete the proof. 
To start with, we have
\begin{equation*}
 \int\limits_{\Omega(x)} \abs{f}^2 + g_\eps(\ud f, \ud f) \,\vol_{g_\eps}
 \leq  \int\limits_{M} \abs{f}^2 + g_\eps(\ud f, \ud f) \,\vol_{g_\eps}
 =\eps^{-1}\langle f, (1- \Delta_{g_\eps}) f \rangle_{L^2(M, g)}
\end{equation*}
Now in view of~\eqref{eq:g diff}
\begin{align*}
 \langle f, - \Delta_{g_\eps} f\rangle_{L^2(M, g)}& \leq 
\abs{\langle f, (H-\lambda)f \rangle} + \abs{\langle f, (\lambda - V_\rho) f \rangle} + \abs{\langle f, \eps^3  S_\eps f\rangle}\\
& \leq \abs{\langle f, \tilde \delta \rangle} + C\eps \norm{f}^2_{L^2(M,g)} + C\eps \int_M g_\eps(\ud f, \ud f) \,\vol_g\,.
\end{align*}
Then since
$\norm{\psi_\eps -\psi}_{W^1(B, g_B)}=\mathcal{O}(\eps)$ we have
\begin{equation*}
 \norm{f}_{L^2(M,g)} = \norm{\partial_s^*(\psi\phi_0-\varphi)}_{L^2(M,g)}\leq \norm{\psi\phi_0-\varphi}_{W^1(M,g)}^2\stackrel{\ref{thm:adiab}}{=}\mathcal{O}(\eps)\,.
\end{equation*}
Together with $\norm{\tilde \delta}_{L^2(M,g)}\leq \norm{\tilde\delta}_\infty \sqrt{\mathrm{Vol}(M)}=\mathcal{O}(\eps)$ this implies $\langle f, (1- \Delta_{g_\eps}) f \rangle=\mathcal{O}(\eps^2)$, which proves the claim.
\end{proof}
\begin{thm}\label{thm:iso}
Let $M$, $\psi$, $\varphi$ and $C$ be as in proposition~\ref{prop:node}. There exists $\eps_0>0$ such that for $\eps<\eps_0$, $\mathcal{N}(\varphi)$ is a smooth submanifold of $M$ that is smoothly isotopic to $\pi^{-1}\mathcal{N}(\psi)$. Consequently the number of connected components of $\mathcal{N}(\varphi)$ equals the number of zeros of $\psi$ and every one of these components is smoothly isotopic through embeddings to the typical fibre of $\pi{:}\,M\to B$.
\end{thm}
\begin{proof}
 We know from proposition~\ref{prop:node} that (in the notation used in the proof there and in corollary~\ref{cor:Hausdorff})
\begin{equation*}
 \mathcal{N}(\varphi)\subset \bigcup_{i\in I} T_{C\eps}(F_{s_i})\,.
\end{equation*}
Hence we may perform the proof by showing that for every $i\in I$, $\mathcal{N}(\varphi)\cap T_{C\eps}(F_{s_i})$ is smoothly isotopic to a fibre.

For a given zero $s_0$ of $\psi$ denote $F_0:= F_{s_0}$ and let $\iota{:}\,F\to F_{s_0}$ be an embedding.
We will show that $\mathcal{N}(\varphi)\cap T_{C\eps}(F_{0})\cong F_0$. First let $v_+$ be the unit tangent vector at $s_0$ pointing in the direction of $\grad \psi$ as in corollary~\ref{cor:Hausdorff}.
Then the map
\begin{align*}
\Phi{:}\,F\times (- 2C\eps, 2C\eps) &\to \pi^{-1}\big((s_0 -2 C\eps, s_0 + 2 C\eps)\big)\\
(y,t)&\mapsto \exp_{\iota(y)}(t v_+^*)
\end{align*}
is a diffeomorphism. It satisfies $\Phi_* \partial_t=\langle v_+,\partial_s\rangle \partial_s^*$, since $\exp_{\iota(y)}(t v_+^*)$ is a horizontal geodesic. Let 
\begin{equation*}
f:=\Phi^* \varphi{:}\, F\times (- 2C\eps, 2C\eps)\to \R\,.
\end{equation*}
By lemma~\ref{lem:deriv} we have
\begin{equation*}
\abs{\partial_t f}=\abs{\partial_s^* \varphi}\geq \big\vert \abs{\partial_s^* \phi_0\psi} - c_1\sqrt\eps\big\vert\,.
\end{equation*}
where $c_1$ is the constant of lemma~\ref{lem:deriv}. Now recall that $\phi_0=\pi^* \mathrm{Vol}(F_s)^{-1/2}$, so $\partial_s^* \phi_0\psi=\partial_s \mathrm{Vol}(F_s)^{-1/2} \psi$. In view of equation~\eqref{eq:H0 circ} one easily checks that
\begin{equation*}
\partial_s^2 \mathrm{Vol}(F_s)^{-1/2} \psi = \mathrm{Vol}(F_s)^{-1/2} H_0 \psi = \mu\mathrm{Vol}(F_s)^{-1/2}\psi\,,
\end{equation*}
Hence $\mathrm{Vol}(F_{s_0})^{-1/2}\psi(s_0)=0$ implies that $\partial_s^*\phi_0 \psi\vert_{F_{s_0}}\neq 0$, since otherwise the differential equation above would imply $\phi_0\psi \equiv 0$.
The value of $\partial_s^*\phi_0 \psi\vert_{F_{s_i}}$ is independent of $\eps$, as both $\phi_0$ and $\psi$ are, and thus for $\eps$ small enough we have \begin{equation*}                                                                                                                                                                                                                                                                                                                                              
 \min \lbrace\abs{\partial_s^*\phi_0 \psi(x)}: \pi(x) \in  [s_0 -2 C\eps, s_0 + 2 C\eps]\rbrace \geq c >0\,,                                                                                                                                                                                                                                                                                                              \end{equation*}
and consequently $\partial_t f \neq 0$ everywhere. Now by proposition~\ref{prop:node} and our choice of $v_+$ we have $f(y,t)>0$ for $t\geq C\eps$ and $f(y,t)<0$ for $t\leq -C\eps$ (so in fact $\partial_t f>0$). Hence for every $y\in F$ there is a unique $t_0(y)\in (-C\eps, C\eps)$ so that $f(y,t_0(y))=0$ and by the implicit function theorem the map $y\mapsto t_0(y)$ is smooth. Hence $f^{-1}(0)=\lbrace (y,t_0(y)): y\in F\rbrace$ is a graph over $F$ and we have
\begin{equation*}
\varphi^{-1}(0)\cap T_{C\eps}(F_0)
=\lbrace\exp_{\iota(y)}(t_0(y) v_+^*):y\in F\rbrace\,, 
\end{equation*}
which shows that $\varphi^{-1}(0)\cap T_{C\eps}(F_0)$ is isotopic through embeddings to $F_0$.
\end{proof}
\subsection{Manifolds with boundary and $\dim B \leq 3$}\label{sect:node gen}
In this section we generalise proposition~\ref{prop:node} to compact manifolds $\pi{:}\,M \to B$ whose base has dimension at most three. Now $M$ may also have a boundary, which requires additional estimates on $\varphi-\psi\phi_0$ near the boundary where both $\varphi$ and $\phi_0$ vanish. It will be a corollary to locating the nodal set that it must intersect the boundary. We obtain a more precise version of this statement in proposition~\ref{prop:boundary}, giving a lower bound on the number of connected components of $\partial M\cap \mathcal{N}(\varphi)$.

The generalisation of proposition~\ref{prop:node} to base dimensions larger than one will require the additional hypothesis that zero is a regular value of $\psi$. As we saw in step one of the proof of~\ref{prop:node} this is automatically satisfied for $\dim B=1$.
Also for $ \dim B>1$ an eigenfunction of $H_0$ and its derivative cannot both vanish on a smooth hypersurface (see e.g.~\cite[lemma 3.4]{GT}). Hence zero is a regular value of $\psi$ if and only if $\psi^{-1}(0)$ is a smooth submanifold of $B$.
\begin{thm}\label{thm:node}
 Assume $d:=\dim B\leq 3$, let $\mu$ be a simple eigenvalue of $H_0$ and ${\psi \in \ker(H_0 - \mu)}$ a real, normalised eigenfunction. Let $\varphi\in \ker (H-\nu)$ and $\theta_d(\eps)$ be as in theorem~\ref{thm:uniform}.
If zero is a regular value of $\psi$ there exist constants $C>0$ and $\eps_0>0$ such that for every $\eps<\eps_0$ and $y\in M\setminus \partial M$ we have:
\begin{equation*}
 \dist_{g_B}(\pi(y),\mathcal{N}(\psi))\geq C\eps \theta_d(\eps) \implies \mathrm{sign}(\varphi(y))=\mathrm{sign}(\psi(\pi(y))\,. 
\end{equation*}
\end{thm}
\begin{proof}
 The proof follows the same steps as that of proposition~\ref{prop:node}. We will stick to the notation introduced there and explain step by step how the arguments may be generalised to the present setting.
\begin{enumerate}[label=\arabic*)]
 \item Because zero is a regular value of $\psi$ and $\mathcal{N}(\psi)$ is obviously compact, there exists a constant $C_0>0$ such that $\abs{v \psi_0}\geq C_0$ on $\mathcal{N}(\psi)$, where $v$ denotes a unit normal of the nodal set of $\psi$.
\item For $\eps$ small enough, $\eps \theta_d(\eps) 2 C_1/C_0$ is smaller than the injectivity radius of $(B,g_B)$. Then for every $x\in \mathcal{N}(\psi)$ the argument of step two in the proof of~\ref{prop:node} applies on the $g$-geodesic $t\mapsto \exp_x(tv)$ and we have (setting $C_2= 2 C_1/C_0$)
\begin{equation*}
 \big\vert\psi\big(\exp_x (C_2 \eps \theta_d(\eps) v\big)\big\vert\geq C_1\eps \theta_d(\eps)\,.
\end{equation*}
\item The argument from~\ref{prop:node} can be applied to the connected components of ${B\setminus \mathcal{N}(\psi)}$ and we obtain $\abs{\psi(x)}> C_1 \eps \theta_d(\eps)$ whenever $\dist(x,\mathcal{N}(\psi))\geq C_2\eps \theta_d(\eps)$.
\item If $\partial M=\varnothing$ the proof concludes just as in the case $\dim B=1$. 

Otherwise we need take into account the behaviour of fibre eigenfunction $\phi_0$ near the boundary, where both $\psi\phi_0$ and $\varphi$ vanish.
\item [4')] Let $D(y):=\dist_{g_{F_{\pi(y)}}}\big(y, \partial F_{\pi(y)}\big)$ be the distance to the boundary measured inside the fibre.
We begin by noting that there exists a positive constant $C_4$ such that for every $y\in M$
\begin{equation*}
\abs{\phi_0(y)/D(y)} \geq C_4 >0\,.
\end{equation*}
This is true because the boundary of $F$ is smooth, and hence by~\cite[lemma 3.4]{GT} the normal derivative of $\phi_0$ on $\partial F_{\pi(y)}$ is non-zero everywhere, which gives a lower bound by compactness of $\partial M$. 
Now since $D(y)\geq \dist_{g_\eps}(y, \partial M)$ we may use corollary~\ref{cor:reg} to obtain the following estimate
\begin{equation*}
 \norm{(\psi\phi_0 - \varphi)/D}_\infty \leq \norm{\psi\phi_0 - \varphi}_{\mathscr{C}^{0,1}(M, g_\eps)}
\leq  C\norm{\psi\phi_0 - \varphi}_{\mathscr{C}^{1}(M, g_\eps)}\leq C_5 \eps \theta_d(\eps)\,.
\end{equation*}
Thus setting $C_1:=C_4^{-1}C_5$ and $C:=C_2=2C_0^{-1}C_1$ completes the proof since for $\dist(\pi(y), \mathcal{N}(\psi))\geq C\eps \theta_d(\eps)$ we have by the previous steps
\begin{equation*}
 \abs{\psi\phi_0(y)/D(y)}\geq \abs{\psi(\pi(y))}C_4 > C_1 C_4 \eps \theta_d(\eps)=C_5\eps \theta_d(\eps)\,.
\end{equation*}
\end{enumerate}
\end{proof}
\begin{prop}\label{prop:boundary}
 Let $\psi$, $\varphi$, $C$ and $\eps_0$ be as in theorem~\ref{thm:node}. Then for all $\eps< \eps_0$ 
\begin{equation*}
\dist\big(\mathcal{N}(\varphi), \pi^{-1}\mathcal{N}(\psi)\big) \leq C \eps \theta_d(\eps)
\end{equation*}
for the Hausdorff distance with respect to $g$. Moreover if $\partial M\neq \varnothing$, then the set ${\mathcal{N}(\varphi)\cap \partial M}$ is non-empty and has at least as many connected components as\\ ${\pi^{-1}\mathcal{N}(\psi)\cap \partial M}$. In particular if $B$ is one-dimensional and $\partial F$ has $k$ connected components, then the number of connected components of ${\mathcal{N}(\varphi)\cap \partial M}$ is at least $k$ times the number of zeros of $\psi$.
\end{prop}
\begin{proof}
The statement on the Hausdorff distance can be proved in the same way as in corollary~\ref{cor:Hausdorff}. That is, one constructs for every $x\in \pi^{-1}\mathcal{N}(\psi)\setminus \partial M$ a curve $\gamma$ with $\gamma(0)=x$, $\varphi(\gamma(t_+))>0$ and $\varphi(\gamma(t_-))<0$.
In order to prove the statement concerning $\mathcal{N}(\varphi)\cap \partial M$ we slightly refine this idea and construct a globally defined map
\begin{equation}\label{eq:proj node}
p{:}\,\pi^{-1}\Big(\lbrace x\in B: \dist\big(x, \mathcal{N}(\psi)\big)\leq C \eps \theta_d(\eps) \rbrace \Big) \to    \pi^{-1}\mathcal{N}(\psi)\,.
\end{equation}
To achieve this let $K$ be a connected component of $\mathcal{N}(\psi)$ and let $v_+$ be the unique unit normal to $K$ pointing into the region where $\psi$ is positive. Set $t_\pm:=\pm C\eps \theta_d(\eps)$, and let $T:=T_{t_+}(K)\subset B$ (in the notation of~\ref{cor:Hausdorff}) be the tubular neighbourhood of $K$ with radius $t_+$.
Parallel transport of $v_+$ along geodesics normal to $K$ defines a vector field $X$ on $T$.
We claim that there exists a lift $\hat X$ of $X$ to $\pi^{-1}(T)$ which is tangent to the boundary, that is we have
\begin{equation*}
\pi_*\hat X = X \,,\qquad
\hat X\vert_{\partial M}\in T\partial M\,.
\end{equation*}
In fact, $\hat X$ can be constructed by covering $T$ by open sets $U_i$, over which $M$ may be trivialised by maps $\Phi_i$ and patching together the vector fields $\Phi_i^*X$, which are clearly tangent to the boundary, using a partition of unity.
The flow of $\hat X$ projects to the flow of $X$ and since $\hat X$ is tangent to $\partial M$, its integral curves exist until their projection reaches the boundary of $T$. Every integral curve of $\hat X$ intersects $\pi^{-1}(K)$ exactly once, so projection along these integral curves defines a smooth map
\begin{equation*}
 p:\pi^{-1}(T) \to \pi^{-1}(K)\,.
\end{equation*}
Repeating this construction for the other components of $\mathcal{N}(\psi)$ defines the projection $p$ of equation~\eqref{eq:proj node}.

Now, because continuous images of connected sets are connected, $\partial M\cap \mathcal{N}(\varphi)$ has at least as many connected components as its image
${p\big(\partial M\cap \mathcal{N}(\varphi)\big)}$. Because $\hat X$ is tangent to the boundary this is contained in $\partial M\cap \pi^{-1}\mathcal{N}(\psi)$.
We conclude the proof by showing that the restriction \begin{equation*}
p{:}\, \partial M \cap \mathcal{N}(\varphi)\to  \partial M\cap \pi^{-1}\mathcal{N}(\psi)                                                                                                                 
\end{equation*}
 is onto. Assume there exists $y\in \partial M\cap \pi^{-1}\mathcal{N}(\psi)$ that is not contained in the image of $p\vert_{\partial M \cap \mathcal{N}(\varphi)}$. Then the integral curve $\gamma$ of $\hat X$ through $y$ does not intersect $\mathcal{N}(\varphi)$. Since the nodal set is closed, there exists also an open neighbourhood $U$ of $\gamma$ in $\pi^{-1}(T)$ with $\mathcal{N}(\varphi)\cap U=\varnothing$. But then there must be a curve in the interior
\begin{equation*}
 \tilde \gamma{:}\,[t_-, t_+] \to U\setminus \partial M\,,                                                                                                                                                                                                                                                                                                                                                                                                                                                                                                                                                                                                                                                                                                                                                                                                                                                                                                                                                                                                      
\end{equation*}
with $\pi(\tilde\gamma(t))=\pi(\gamma(t))$. It follows from theorem~\ref{thm:node} that ${\varphi(\tilde\gamma(t_+))>0}$, ${\varphi(\tilde\gamma(t_-))<0}$ and this contradicts the fact that $\tilde \gamma \cap \mathcal{N}(\varphi)=\varnothing$, so such a point point $y\in \partial M\cap \pi^{-1}\mathcal{N}(\psi)$ cannot exist. 
\end{proof}

\subsection*{Acknowledgements}
The author thanks David Krej{\v{c}}i{\v{r}}{\'\i}k and Stefan Teufel for stimulating discussions and the anonymous referee for valuable comments. 
Financial support from the European Research Council under the European Community’s Seventh Framework Programme (FP7/2007-2013 Grant Agreement MNIQS 258023) and the German Science Foundation (DFG) within the SFB Transregio 71 is gratefully acknowledged.

\bibliographystyle{abbrv}
\bibliography{bibliography}
\end{document}